\documentclass[11pt]{amsart}  

\usepackage{lineno}
\usepackage[top=1.4in, bottom=1.4in, left=1.5in, right=1.5in]{geometry}

\usepackage{amsrefs}
\usepackage[all]{xy}
\usepackage{syntonly}
\usepackage{enumerate}
\usepackage{hyperref}
\usepackage{amsfonts}
\usepackage{amssymb}
\usepackage{amsmath}
\usepackage{amsthm}
\usepackage{enumitem}

\usepackage{tikz}

\usepackage{graphics}
\usepackage{graphicx}
\usepackage{color}
\usepackage{comment}

\newtheorem{theorem}{Theorem}
\newtheorem{corollary}[theorem]{Corollary}

\newtheorem{definition}[theorem]{Definition}
\newtheorem{lemma}[theorem]{Lemma}

\newtheorem{nonGlobalClaim}{Claim}[theorem]  

\newtheorem{observation}[theorem]{Observation}
\newtheorem{question}[theorem]{Question}

\newtheorem{convention}[theorem]{Convention}
\newtheorem{remark}[theorem]{Remark}

\specialcomment{com}{ \color{red} }{\color{black}} 

\newcommand{\lav}[1]{\Diamond_{\text{Lav}}( #1)}

\newcommand{\lavminus}[1] {
  \Diamond^-_{\text{Lav}}( #1)}

\newcommand{\lavplus}[1] {
  \Diamond^+_{\text{Lav}}( #1)}

\begin{document}
\title[Prevalence of Generic Laver Diamond]{Prevalence of Generic Laver Diamond}

\author{Sean D. Cox}
\email{scox9@vcu.edu}
\address{
Department of Mathematics and Applied Mathematics \\
Virginia Commonwealth University \\
1015 Floyd Avenue \\
Richmond, Virginia 23284, USA 
}

\thanks{In memory of Richard Laver, 1942-2012}

\thanks{Part of this work was done while the author participated in the Thematic Program on Forcing and its Applications at the Fields Institute, which was partially supported from NSF grant DMS-1162052.}


\begin{abstract}
Viale~\cite{Viale_GuessingModel} introduced the notion of Generic Laver Diamond at $\kappa$---which we denote $\Diamond_{\text{Lav}}(\kappa)$---asserting the existence of a single function from $\kappa \to H_\kappa$ that behaves much like a supercompact Laver function, except with generic elementary embeddings rather than internal embeddings.  Viale proved that the Proper Forcing Axiom (PFA) implies $\lav{\omega_2}$.  We strengthen his theorem by weakening the hypothesis to a statement strictly weaker than PFA.  We also show that the principle $\Diamond_{\text{Lav}}(\kappa)$ provides a uniform, simple construction of 2-cardinal diamonds, and prove that $\lav{\kappa}$ is quite prevalent in models of set theory; in particular:
\begin{enumerate}
 \item  $L$ satisfies $\lavplus{\kappa}$ whenever $\kappa$ is a successor cardinal, or when the appropriate version of Chang's Conjecture fails.
  \item For any successor cardinal $\kappa$, there is a $\kappa$-directed closed class forcing---namely, the forcing from Friedman-Holy~\cite{MR2860182}---that forces $\lav{\kappa}$.
\end{enumerate} 
\end{abstract}

\maketitle


Prediction principles have been central topics in set theory ever since Jensen introduced the $\Diamond$ principle in the 1960s.  Not only does $\Diamond$ hold in canonical inner models such as $L$, it also is frequently introduced by forcing; for example, adding a Cohen subset of $\kappa$ introduces a $\Diamond_\kappa$ sequence.  The $\Diamond$ principle is frequently used to prove one direction of an independence result.\footnote{e.g. independence of  Suslin's Hypothesis, Whitehead's Problem, existence of non-inner automorphisms of the Calkin algebra, Naimark's Problem, and many others.}   Two-cardinal variations of $\Diamond$---i.e. versions of $\Diamond$ which guess subsets of some fixed $\wp_\kappa(\lambda)$, rather than just guessing subsets of $\kappa$---were introduced by Jech~\cite{MR0325397}.  Donder-Matet~\cite{MR1239715} (with a correction by  Shioya~\cite{MR1738689}) proved that mild cardinal arithmetic assumptions are enough to guarantee such 2-cardinal versions of $\Diamond$; for example, $\kappa^{<\kappa} = \kappa$ implies that $\Diamond(\kappa, \lambda)$ holds for all $\lambda \ge \kappa^+$.  

Laver~\cite{MR0472529} proved that if $\kappa$ is a supercompact cardinal, then there is a function $F: \kappa \to H_\kappa$ which essentially behaves like a universal $\Diamond_\kappa$ sequence with respect to supercompactness measures (rather than merely with respect to the club filter, as is the case with usual $\Diamond_\kappa$ and its 2-cardinal variants).  For this reason it is commonly called a \emph{Supercompact Laver Diamond/Function}, and notably appears in the consistency proofs of the Proper Forcing Axiom, Martin's Maximum, and an indestructibly supercompact cardinal.   

The concept of a Laver function has been generalized in two distinct directions; unfortunately both are referred to as ``Generalized Laver Diamond" in the literature:
\begin{enumerate}
 \item Generalizations of (supercompact) Laver Diamond to some other large cardinal notions (e.g.\ strong cardinals) have been developed (e.g. \cite{MR987765}, \cite{MR1779746}, \cite{Hamkins_LaverDiamond}, and \cite{MR2026390}).  
 \item Viale~\cite{Viale_GuessingModel} generalized the notion of Laver Diamond so that it makes sense at \emph{successor} cardinals. 
\end{enumerate}

 To distinguish Viale's version from the large cardinal versions mentioned above, we will refer to Viale's version as ``Generic" Laver Diamond at $\kappa$, denoted by $\lav{\kappa}$.  This is a function from $\kappa \to H_\kappa$ which behaves somewhat like a Laver function, except with generic rather than internal ultrapowers (it can also be defined without reference to forcing; see Section \ref{sec_LaverFunctions}).  In particular, $\lav{\kappa}$ can hold when $\kappa$ is a \emph{successor} cardinal, and $\lav{\kappa}$ provides a particularly elegant, uniform way to produce two-cardinal diamond sequences; see Section \ref{sec_TwoCardinalDiamond} for such a construction.  Viale's main result was that the Proper Forcing Axiom (PFA) implies $\lav{\omega_2}$.  We strengthen this theorem by weakening the hypothesis to a statement strictly weaker than PFA; moreover our proof is completely elementary and does not make use of the coding of $H_{\omega_2}$ which appeared in Viale's proof.\footnote{The hypothesis that $\mathfrak{c} = \omega_2$ and $\text{GIC}_{\omega_1}$ is a stationary class is a consequence of PFA (by Viale-Wei\ss~\cite{VW_ISP} and Todorcevic~\cite{MR980949}), but is strictly weaker than PFA; see Remark \ref{rem_ISP_Weaker}.  In particular it does not imply the existence of the Caicedo-Velickovic coding.}  
\begin{theorem}\label{thm_StrengthenViale}
If $\mathfrak{c} = \omega_2$ and the class $\text{GIC}_{\omega_1}$ of $\omega_1$-guessing, internally club sets is stationary, then $\lav{\omega_2}$ holds.

More generally:  whenever $\Gamma \subset \wp_{\kappa}(V)$ is a $\Pi_1(V)$ definable stationary class in some parameter from $H_{\kappa^+}$ such that:
\begin{itemize}
 \item  $\Gamma$ projects downward; and
 \item  $\Gamma$ satisfies the Isomorphism Property
\end{itemize}
then there is a $\lav{\Gamma}$ function. 
\end{theorem}
Our proof of Theorem \ref{thm_StrengthenViale} gives an explicit construction of the $\lav{\Gamma}$ function; namely, the function defined recursively in Claims \ref{clm_AtMostOne} and  \ref{clm_ExplicitConstructionLaver}. 

We also prove---as is the case with the weaker $\Diamond_\kappa$ and its 2-cardinal variants---that the principle $\lav{\kappa}$ is quite prevalent in models of set theory:  
\begin{theorem}\label{thm_MainTheoremStatCondens}
If Stationary Condensation holds at $\kappa$ and $\kappa$ is a successor cardinal, then $\lav{\kappa}$ holds.  More generally, if the appropriate version of weak Chang's Conjecture fails at $\kappa$, then $\lav{\kappa}$ holds.
\end{theorem}
Combining this theorem with the results of Friedman-Holy~\cite{MR2860182} yields:
\begin{corollary}
If $\kappa$ is a successor cardinal, then there is a $\kappa$-directed closed class forcing extension that models $\lav{\kappa}$.
\end{corollary}

We also prove:
\begin{theorem}\label{thm_L_LaverPlus}
$L$ satisfies $\lavplus{\kappa}$ whenever $\kappa$ is a successor cardinal, or whenever the appropriate version of weak Chang's Conjecture fails at $\kappa$.  By Theorem \ref{thm_GeneralizeKunenDiamondMinus} $\lav{\kappa}$ also holds for such $\kappa$.
\end{theorem}

The paper is structured as follows:  Section \ref{sec_LaverFunctions} provides the definitions and basic facts about $\lav{\kappa}$ and its variants; Section \ref{sec_GeneralizeKunen} proves that $\lav{\kappa}$ is equivalent to $\lavminus{\kappa}$; Section \ref{sec_CondensationChang} proves Theorems \ref{thm_MainTheoremStatCondens} and \ref{thm_L_LaverPlus}; Section \ref{sec_StrengthenViale} proves Theorem \ref{thm_StrengthenViale}; and Section \ref{sec_Conclusion} concludes with a question.  

We will use the following notation throughout the paper:
\begin{definition}
Suppose $X$ is a set and $\Big( X, \in \restriction (X \times X) \Big)$ is extensional.\footnote{e.g. if $X \prec (H_\theta, \in)$ for some $\theta$.} Then $H_X$ denotes the transitive collapse of $X$ and $\sigma_X: H_X \to X$ denotes the \textbf{inverse} of the Mostowski collapsing map of $X$.
\end{definition}

We also will use the following convention:
\begin{convention}
If $p_1, \dots, p_n$ are each elements of some $H_\theta$, then each $p_i$ in the structure $(H_\theta, \in, p_1, \dots, p_n)$ will be understood to be the natural interpretation of a constant symbol $\dot{p}_i$.  
\end{convention}

The author would like to thank the anonymous referee for many helpful suggestions, especially regarding the structure of the proof of Theorem \ref{thm_GeneralizeKunenDiamondMinus}.

\section{$\lav{\Gamma}$, $\lavplus{\Gamma}$, and $\lavminus{\Gamma}$ }\label{sec_LaverFunctions}

We recall the definition of Generic Laver Diamond from Viale~\cite{Viale_GuessingModel}:  a function $F:\kappa \to H_\kappa$ is a Generic Laver function iff for every set $b$ and for every sufficiently large regular $\theta$, there are stationarily many $M \in \wp_\kappa(H_\theta)$ such that $b \in M$ and the Mostowski collapse of $M$ sends $b$ to $F(M \cap \kappa)$.\footnote{Viale's formulation was actually more similar to the characterization in Lemma \ref{lem_LaverFcn_As_GenericLargeCardinal} below.}  The role of the class $\Gamma$ in the following definition is to allow for refinements; similarly to the way that $\Diamond(S)$ refines $\Diamond_\kappa$ (where $S$ is a stationary subset of $\kappa$).  
\begin{definition}\label{def_MainDef}
Suppose $\kappa$ is a regular uncountable cardinal and $F: \kappa \to V$ is a function.  If $b$ is a set and $\theta$ is a regular uncountable cardinal, define $G^{b,\theta}_{=F}$ to be the set of those $M$ such that:
\begin{enumerate}[label=(\alph*)]
 \item $b \in M \prec (H_{\theta}, \in)$
 \item $M \cap \kappa \in \kappa$
 \item\label{item_M_guesses} $\sigma_M^{-1}(b) = F(M \cap \kappa)$ (recall that $\sigma_M:H_M \to M$ denotes the inverse of the Mostowski collapse of $M$)
\end{enumerate}
For any class $\Gamma$, we say that $F$ is a $\lav{\Gamma}$ function iff  for every set $b$ and every sufficiently large regular $\theta$, the following set is stationary:
\begin{equation*}
\Gamma \cap G^{b,\theta}_{=F} 
\end{equation*}
We say that $\Diamond_{\text{Lav}}(\Gamma)$ holds iff there is a $\Diamond_{\text{Lav}}(\Gamma)$ function.  
\end{definition}

We also define variants of Generic Laver Diamond which are analogous to what Kunen~\cite{MR597342} calls $\Diamond^-$ and $\Diamond^+$:
\begin{definition}\label{def_LavMinus}
Suppose $\kappa$ is a regular uncountable cardinal and $F: \kappa \to V$ is a function.  For each set $b$ and regular uncountable $\theta$, the set $G^{b,\theta}_{\in F}$ is defined the same way that $G^{b,\theta}_{=F}$ was defined in Definition \ref{def_MainDef} except requirement \ref{item_M_guesses} is replaced with the following requirement:
\begin{enumerate}
 \item[\ref{item_M_guesses}'] $\sigma_M^{-1}(b) \in F(M \cap \kappa)$
\end{enumerate}
For any class $\Gamma$, we say that $F$ is a $\lavminus{\Gamma}$ (resp. $\Diamond^+_{\text{Lav}}(\Gamma)$) function iff
\begin{enumerate}[label=(\Alph*)]
 \item\label{item_Lavminus_guessing} For every set $b$ and all sufficiently large regular $\theta$:  $\Gamma \cap G^{b,\theta}_{\in F}$ is stationary (resp. contains all but nonstationarily many elements of $\Gamma \cap \wp(H_\theta)$)
 \item\label{item_Lavminus_small}  $|F(\alpha)| \le |\alpha|$ for every $\alpha < \kappa$. 
\end{enumerate}
We say $\lavminus{\Gamma}$ (resp. $\lavplus{\Gamma}$) holds iff there exists a $\lavminus{\Gamma}$ (resp. $\lavplus{\Gamma}$) function.
\end{definition}

\begin{remark}
In the special case where $\Gamma$ is the natural class $\wp_\kappa(V)$, we will just write $\lav{\kappa}$ to mean $\lav{\Gamma}$, and similarly for $\lavminus{\kappa}$ and $\lavplus{\kappa}$.
\end{remark}

There are obvious local variations of Definitions \ref{def_MainDef} and \ref{def_LavMinus}, but we will not make use of such local variations in this paper.\footnote{Namely, one could weaken $\lav{\Gamma}$ by restricting attention to only those $b$ up to some fixed cardinality.}

\begin{remark}\label{remark_ReplaceFunction}
If $\Gamma \subseteq \wp_\kappa(V)$ and $F$ is a $\Diamond_{\text{Lav}}(\Gamma)$ function, then so is
\begin{equation*}
F \restriction \{ \alpha < \kappa \ | \ F(\alpha) \in H_\kappa  \}
\end{equation*}
\end{remark}
\begin{proof}
Let $b$ be any set.  Let $\theta$ be a sufficiently large regular cardinal so that $\Gamma \cap G^{b,\theta}_{=F}$ is stationary.  Since $\Gamma \subseteq \wp_\kappa(V)$ then $H_M$ is an element of $H_\kappa$ for every $M \in \Gamma \cap G^{b,\theta}_{=F}$; so $F(M \cap \kappa) = \sigma_M^{-1}(b) \in H_M \in H_\kappa$ for any such $M$.
\end{proof}

\begin{remark}
If $\Gamma \subseteq \Gamma'$ and $\lav{\Gamma}$ holds, then so does $\lav{\Gamma'}$.  This is analogous to the trivial fact that $\Diamond(S)$ implies $\Diamond(S')$ whenever $S \subseteq S'$ are stationary subsets of $\kappa$.
\end{remark}

Definition \ref{def_MainDef} can also be rephrased in terms of generic elementary embeddings, a characterization which more closely resembles the definition of supercompact Laver functions.  This is the version which appeared in Viale~\cite{Viale_GuessingModel} (albeit in the presence of Woodin cardinals).  While the following lemma is technically a second-order scheme, it is also possible to obtain a first-order definition involving generic embeddings.\footnote{Namely:  in the statement of Lemma \ref{lem_LaverFcn_As_GenericLargeCardinal}, if one replaces ``there is a generic elementary embedding $j: V \to N$"  with ``there is a normal ideal $\mathcal{I}$ such that $\Vdash_{(\mathcal{I}^+, \subseteq)}$  the ultrapower embedding $j: H^V_{(2^\theta)^+} \to_{\dot{G}} \text{ult}(H^V_{(2^\theta)^+}, \dot{G})$ has critical point $\kappa$ and has the following properties \dots" then the resulting statement is first order, and the proof is similar to the proof given here. } \begin{lemma}\label{lem_LaverFcn_As_GenericLargeCardinal}
A function $F: \kappa \to V$ is a $\Diamond_{\text{Lav}}(\Gamma)$-function if and only if for every $b$ and every sufficiently large regular $\theta$,  there is a generic elementary embedding $j: V \to N$ with critical point $\kappa$ such that:
\begin{enumerate}
 \item $H^V_\theta$ is an element of the (transitivized) wellfounded part of $N$;
 \item $j[H^V_\theta] \in N$;
 \item $j(F)(\kappa) = b$; and
 \item $N \models j[H^V_\theta] \in j(\Gamma)$
\end{enumerate}
\end{lemma}
\begin{proof}
Assume $F:\kappa \to H_\kappa$ is a $\lav{\Gamma}$ function.  Fix some $b$ and let $\theta$ be sufficiently large so that $S:= \Gamma \cap G^{b,\theta}_{=F}$ is stationary (see Definition \ref{def_MainDef} for the meaning of $G^{b,\theta}_{=F}$).  Let $\mathcal{I}$ be the restriction of the nonstationary ideal to the stationary set $S$.  Let $\mathcal{U}$ be $(V, \wp(S)/\mathcal{I})$-generic and $j:V \to_{\mathcal{U}} N_{\mathcal{U}}$ the generic ultrapower embedding; here $N_{\mathcal{U}}$ may not be wellfounded, but standard applications of Los' Theorem imply that $H^V_\theta$ is an element of its (transitivized) wellfounded part and moreover:
\begin{itemize}
 \item $j[H^V_\theta] = [\text{id}]_{\mathcal{U}} \in j(\Gamma)$  (since $S \subseteq \Gamma$)
 \item $j(F)(\kappa) = b$ (since $S \subseteq G^{b,\theta}_{=F}$)
\end{itemize}

Conversely, suppose $F$ is a function from $\kappa \to H_\kappa$ and for every $b$ and sufficiently large $\theta$ there is a generic embedding $j=j_{b,H_\theta}:V \to N$ satisfying the requirements listed in the statement of the lemma.  Let $\mathfrak{A} = (H_\theta, \in, \{b\}, \dots)$ be arbitrary and let $M':=j[H^V_\theta]$.  Then $\sigma_{M'} = j \restriction H^V_\theta$ and $N$ models the following facts about $M'$:
\begin{itemize}
 \item $M' \prec j(\mathfrak{A})$
 \item $j(F)( M' \cap j(\kappa)) = b =  \sigma_{M'}^{-1}(j(b))$ 
 \item $M' \in j(\Gamma)$
\end{itemize}
By elementarity of $j$, $V$ believes there is an $M$ such that $M \prec \mathcal{A}$, $F(M \cap \kappa) = \sigma_{M}^{-1}(b)$, and $M \in \Gamma$.   
\end{proof}

We will use the following definition; many natural classes (e.g.\ classes of internally approachable structures, etc.) have this property:
\begin{definition}\label{def_DownwardProjecting}
A class \emph{$\Gamma$ projects downward} iff whenever $X \in \Gamma$ and $\theta \le \text{sup}(X \cap ORD)$ is a regular uncountable cardinal, then $X \cap H_\theta \in \Gamma$. 
\end{definition}

\section{Equivalence of $\lav{\Gamma}$ with $\lavminus{\Gamma}$}\label{sec_GeneralizeKunen}

Theorem \ref{thm_GeneralizeKunenDiamondMinus} below---which is a generalization of Kunen's proof that $\Diamond^-$ is equivalent to $\Diamond$ (cf. Theorem 7.14 of \cite{MR597342})---is crucial to the proof of Theorem \ref{thm_MainTheoremStatCondens}.  
\begin{theorem}\label{thm_GeneralizeKunenDiamondMinus}
Let $\Gamma$ be a downward-projecting class (in the sense of Definition \ref{def_DownwardProjecting}) and $\kappa$ a regular uncountable cardinal, and suppose that $M \cap \kappa \in \kappa$ for every $M \in \Gamma$.  Then
\begin{equation*}
\Diamond_{\text{Lav}}^-(\Gamma) \iff \Diamond_{\text{Lav}}(\Gamma)
\end{equation*}
\end{theorem}
\begin{proof}
The $\Leftarrow$ direction is trivial, since if $F:\kappa \to H_\kappa$ is a $\lav{\Gamma}$ function then easily $\alpha \mapsto \{ F(\alpha) \}$ is a $\lavminus{\Gamma}$ function.  Now we prove the $\Rightarrow$ direction.  Suppose $F:\kappa \to V$ witnesses $\Diamond_{\text{Lav}}^-(\Gamma)$.  Let $\langle z_{\alpha, i} \ | \  i < \alpha  \rangle$ be an enumeration of $F(\alpha)$ for each $\alpha < \kappa$.  For each $i < \kappa$ define a function $F_i: \kappa \to H_\kappa$ as follows:  if $z_{\alpha,i}$ is a function with $i$ in its domain, then let  $F_i(\alpha):=z_{\alpha,i}(i)$; otherwise $F_i(\alpha):=\emptyset$.   We finish the proof by showing:
\begin{nonGlobalClaim}
There is some $i < \kappa$ such that $F_i$ is a $\lav{\Gamma}$ function.
\end{nonGlobalClaim}
\begin{proof}
Suppose not; so for every $i < \kappa$ there is some $b_i$ and some algebra $\mathfrak{A}_i = (H_{\theta_i}, \in, b_i )$ such that:
\begin{equation}\label{eq_ContradictionAssumption}
\forall M \in \Gamma \ M \prec \mathfrak{A}_i \implies \sigma_M^{-1}(b_i) \ne F_i(M \cap \kappa)
\end{equation}

Let $B:= \langle b_i \ | \ i < \kappa \rangle$ and fix some regular $\Omega$ such that $B$ and $\langle \mathfrak{A}_i, b_i \ | \ i < \kappa \rangle$ are elements of $H_{\Omega}$.  Let $T$ denote the set of $M \in \Gamma$ such that $M \prec \mathfrak{B}:= (H_{\Omega}, \in, B, \vec{\mathfrak{A}}, \vec{b})$ and $\sigma_M^{-1}(B) \in F(M \cap \kappa) $; $T$ is stationary because $F$ is a $\lavminus{\Gamma}$ function.  Now $\sigma_M^{-1}(B) \in F(M \cap \kappa) = \{ z_{M \cap \kappa, i} \ | \ i < M \cap \kappa \} $, so there is some $i_M < M \cap \kappa$ such that $\sigma_M^{-1}(B) = z_{M \cap \kappa, i_M}$.  Since $i_M \in M$ for all $M \in T$ then by Fodor's Lemma there is a stationary $T' \subseteq T$ and some fixed $\hat{i}$ such that $i_M = \hat{i}$ for all $M \in T'$.   Fix an $M  \in T'$.  First observe that 
\begin{equation}\label{eq_PreimageOfB}
\sigma_M^{-1}(B) = \sigma_M^{-1}(\langle b_i \ | \ i < \kappa \rangle) = \langle \sigma_M^{-1}(b_i) \ | \ i \in M \cap \kappa \rangle
\end{equation} 
so in particular $\sigma_M^{-1}(B)$ is a function with $\hat{i}$ in its domain.  Now 
\begin{equation}
 \sigma_M^{-1}(B) = z_{M \cap \kappa, i_M} =  z_{M \cap \kappa, \hat{i}}
\end{equation}
and so in particular
\begin{equation}\label{eq_B_matches_z}
  \sigma_M^{-1}(B) (\hat{i}) = z_{M \cap \kappa, \hat{i}}(\hat{i})
\end{equation}
So (\ref{eq_B_matches_z}) and (\ref{eq_PreimageOfB}) imply:
\begin{equation}\label{eq_SigmaM_matchesAgain}
  \sigma_M^{-1}(b_{\hat{i}}) = z_{M \cap \kappa, \hat{i}}(\hat{i})
\end{equation}
Now by the definition of $F_{\hat{i}}$:
\begin{equation}\label{eq_DefinitionF_i_implies}
F_{\hat{i}}(M \cap \kappa) = z_{M \cap \kappa, \hat{i}}(\hat{i})
\end{equation}
Finally (\ref{eq_DefinitionF_i_implies}) and (\ref{eq_SigmaM_matchesAgain}) imply:
\begin{equation}\label{eq_MainEquality}
  F_{\hat{i}}(M \cap \kappa) = \sigma_M^{-1}(b_{\hat{i}})
\end{equation}
On the other hand, since $M \prec \mathfrak{B}$ and $\hat{i} \in M$ then $M_{\hat{i}}:= M \cap H_{\theta_{\hat{i}}} \prec \mathfrak{A}_{\hat{i}}$; moreover $M_{\hat{i}} \in \Gamma$ by our assumption that $\Gamma$ projects downward.   So (\ref{eq_ContradictionAssumption}) implies:
\begin{equation}\label{eq_NonEqualForCutdown}
F_{\hat{i}}(M_{\hat{i}} \cap \kappa)  \ne \sigma_{M_{\hat{i}}}^{-1}(b_{\hat{i}}) 
\end{equation}
But $M_{\hat{i}} \cap \kappa = M \cap \kappa $ and $\sigma_M^{-1}(b_{\hat{i}}) = \sigma_{M_{\hat{i}}}^{-1}(b_{\hat{i}})$, so (\ref{eq_NonEqualForCutdown}) implies:
\begin{equation}
F_{\hat{i}}(M \cap \kappa) \ne \sigma_M^{-1}(b_{\hat{i}})
\end{equation}
which contradicts (\ref{eq_MainEquality}) and completes the proof of the claim.  
\end{proof}
\end{proof}

\section{A simple proof of 2-cardinal Diamond}\label{sec_TwoCardinalDiamond}

Jech~\cite{MR0325397} introduced a 2-cardinal Diamond principle, which is a guessing principle for subsets of $\wp_\kappa(\lambda)$.  Precisely, $\Diamond(\kappa, \lambda)$ asserts the existence of a function $\langle A_z \ | \ z \in \wp_\kappa(\lambda)  \rangle$ such that for every $A \subset \lambda$, the following set is stationary in $\wp_\kappa(\lambda)$:
\begin{equation*}
\{  z \in \wp_\kappa(\lambda) \ | \ A \cap z = A_z     \}
\end{equation*}
Viale's principle $\lav{\kappa}$ easily implies that $\kappa^{<\kappa} = \kappa$, which in turn---using a theorem of Donder-Matet~\cite{MR1024901} and a correction by Shioya~\cite{MR1738689}---implies that $\Diamond(\kappa, \lambda)$ holds for all $\lambda > \kappa$.\footnote{Donder-Matet~\cite{MR1024901}, with the correction by Shioya~\cite{MR1738689}, proved that $\Diamond(\kappa, \lambda)$ holds \emph{whenever} $2^{<\kappa} < \lambda$.}  However, there is an especially simple, direct proof of the implication
\begin{equation*}
\lav{\kappa} \implies \forall \lambda > \kappa \ \Diamond(\kappa, \lambda)
\end{equation*}
which does not require going through the theorem of Donder-Matet and Shioya (though of course their proof is much more general, as it assumes only that $2^{<\kappa}<\lambda$).  Suppose $F:\kappa \to H_\kappa$ witnesses $\lav{\kappa}$, and let $\lambda \ge \kappa$ be any cardinal.  For each (extensional) $M \in P_\kappa(\lambda)$, recall that $\sigma_M$ denotes the inverse of the Mostowski collapsing map of $M$, and set
\begin{equation*}
A_M:= \sigma_M[F(M \cap \kappa)]
\end{equation*}
Then $\langle A_M \ | \ M \in P_\kappa(\lambda) \rangle$ is a $\Diamond_{\kappa, \lambda}$ sequence: let $A \subseteq \lambda$.  By Laverness of $F$, there is a stationary set $S'$ of $M' \prec (H_{(2^\lambda)^+}, \in, \{ A \})$ such that $\sigma_{M'}^{-1}(A) = F(M' \cap \kappa)$.  Then for any $M' \in S'$, setting $M:= M' \cap \lambda$ we have:
\begin{align*}
A \cap M = A \cap (M' \cap H_\lambda) = \sigma_{M'}(F(M' \cap \kappa)) \cap (M' \cap H_\lambda) \\
= \sigma_{M'} [ F(M' \cap \kappa) ] = \sigma_M[F(M \cap \kappa)] = A_M
\end{align*}
Thus, setting
\begin{equation*}
S:= \{ M' \cap \lambda \ | \ M' \in S'  \}
\end{equation*}
we have that $A \cap M = A_M$ for every $M \in S$.

\section{$\lav{\kappa}$, Condensation, and weak Chang's Conjecture}\label{sec_CondensationChang}

We will describe a natural attempt to define a $\lavminus{\kappa}$ function, which often works in the presence of Condensation and/or the appropriate failure of Chang's Conjecture.  For regular uncountable cardinals $\kappa <\theta$, the Chang's Conjecture $(\theta, \kappa) \twoheadrightarrow (\kappa, < \kappa)$ means that for every first order structure $\mathfrak{A}$ on $\theta$ (in a countable language) there is an $M \prec \mathfrak{A}$ such that $|M| = \kappa$ and $|M \cap \kappa| < \kappa$; this is equivalent to saying that
\begin{equation*}
\{  M \subset \theta \ | \  |M|=\kappa \text{ and } |M \cap \kappa| < \kappa    \}
\end{equation*}
is a (weakly) stationary set.  \emph{Weak Chang's Conjecture holds at $\kappa,\theta$}---abbreviated $\text{wCC}(\kappa,\theta)$---holds iff for every every first-order structure $\mathfrak{A}=(\theta, \in , \dots)$ in a countable language, there are stationarily many $\alpha < \kappa$ such that 
\begin{equation}\label{eq_TallStructuresAtAlpha}
\text{sup} \{ \text{ot}(X \cap \theta) \ | \  X \prec \mathfrak{A} \text{ and } X \cap \kappa = \alpha     \} \ge \alpha^+
\end{equation}
It is easy to see that decreasing the parameter $\theta$ in $\text{wCC}(\kappa, \theta)$ increases the strength; i.e. $\text{wCC}(\kappa, \theta) \implies \text{wCC}(\kappa, \theta')$ whenever $\theta \le \theta'$.  Thus the strongest is when $\theta = \kappa^+$; the  principle  $\text{wCC}(\kappa, \kappa^+)$ is the well-known ``weak Chang's Conjecture at $\kappa$", denoted $\text{wCC}(\kappa)$ in Definition 1.6 of Donder-Levinski~\cite{MR1024901}.  
\begin{remark}
Under $V = L$, a cardinal $\kappa$ is ineffable iff $\text{wCC}(\kappa)$ holds.  See Corollary 1.13 of Donder-Levinski~\cite{MR1024901}.
\end{remark}
\begin{remark}
Standard techniques (e.g. as in Foreman-Magidor~\cite{MR1359154}) enable reformulation of, say, $\text{wCC}(\kappa, \theta)$ which only refers to a single structure.  For example,   fixing a wellorder $\Delta$ of $H_{(2^\theta)^+}$, the principle $\text{wCC}(\kappa, \theta)$ is equivalent to saying there are  stationarily many $\alpha < \kappa$ such that
\begin{equation*}
\text{sup} \{ \text{otp}(X \cap \theta) \ | \ X \prec  (H_{(2^\theta)^+}, \in, \Delta) \text{ and } X \cap \kappa = \alpha  \} \ge \alpha^+
\end{equation*}
\end{remark}

For a regular uncountable $\kappa$ and a class $\Gamma \subset \wp_\kappa(V)$, we say that \emph{$\Gamma$ is $< ORD$ stationary} (resp. \emph{$\Gamma$ is $< ORD$ club}) if $\Gamma \cap \wp_\kappa(H_\theta)$ is stationary (resp. contains a club) for all sufficiently large $\theta$.  Note that $\lav{\Gamma}$ trivially implies that $\Gamma$ is $< ORD$ stationary.

\begin{definition}\label{def_SmallBelowKappa}
Let $\Gamma \subseteq \wp_\kappa(V)$.  We say that $\text{wCC}(\Gamma)$ holds iff 
\begin{enumerate}[label=(\alph*)]
 \item $\Gamma$ is $<ORD$ stationary; and
 \item There are stationarily many $\alpha < \kappa$ such that 
 \begin{equation}\label{eq_Tall}
 \text{sup} \{  \text{otp}(X \cap ORD) \ | \ X \in \Gamma \ \wedge \ X \prec_{\Sigma_1} (V,\in) \ \wedge \  X \cap \kappa = \alpha           \} \ge \alpha^+
 \end{equation}
\end{enumerate}
The principle $\text{wCC}^*(\Gamma)$ is defined similarly, except (\ref{eq_Tall}) is replaced by:
\begin{equation}
\Big| \bigcup  \{ H_X \ | \  X \in \Gamma \ \wedge \ X \prec_{\Sigma_1} (V,\in) \ \wedge \  X \cap \kappa = \alpha             \}  \Big| \ge \alpha^+
\end{equation}
\end{definition}

\begin{observation}\label{obs_SmallnessImpliesShortness}
If $\Gamma \subseteq \wp_\kappa(V)$, the principle $\text{wCC}(\Gamma)$ implies the principle $\text{wCC}^*(\Gamma)$, since
\begin{align*}
\{ \text{otp}(X \cap \text{ORD}) \ | \ X \in \Gamma  \text{ and } X \cap \kappa = \alpha   \} \\
= \{ \text{height}(H_{X \cap ORD}) \ | \ X \in \Gamma  \text{ and } X \cap \kappa = \alpha     \}
\end{align*}
\end{observation}

\begin{definition}\label{def_AttemptLaverMinus}
Suppose $\kappa$ is regular and uncountable and that $\Gamma \subseteq \{  X \ | \ X \cap \kappa \in \kappa \}$.   For each $\alpha < \kappa$ define:
\begin{equation}\label{eq_A_alpha}
A^{\Gamma}_\alpha:=  \bigcup \{ H_M \ | \ M \in \Gamma \text{ and } M \cap \kappa =  \alpha    \}
\end{equation}
Define the map $F^{\Gamma}$ with domain $\kappa$ by $\alpha \mapsto A^{\Gamma}_\alpha$.
\end{definition}

\begin{lemma}\label{lem_ConstructLaver}
Suppose $\Gamma$ is a $< ORD$-stationary (resp. club) subclass of $\wp_\kappa(V)$, and let $F^{\Gamma}$ be as in Definition \ref{def_AttemptLaverMinus}.  Then:  
\begin{enumerate}
 \item\label{item_SatisfiesGuessing} $F^{\Gamma}$ satisfies Requirement \ref{item_Lavminus_guessing} in the Definition \ref{def_LavMinus} of a $\lavminus{\Gamma}$ function.
 \item\label{item_SetCatchesAll} There is some $\theta$ such that $F^{\Gamma} = F^{\Gamma \cap \wp_\kappa(H_\theta)}$
 \item\label{item_Short} Suppose $\text{wCC}^*(\Gamma)$ fails.  Then $F^{\Gamma}$ also satisfies requirement \ref{item_Lavminus_small} of Definition \ref{def_LavMinus}, and thus $F^{\Gamma}$ is a $\lavminus{\Gamma}$ (resp. $\lavplus{\Gamma}$) function.
\end{enumerate}
\end{lemma}
\begin{proof}
Consider any set $b$ and any structure $\mathfrak{A} = (H_\theta, \in, \{ b \}, \dots)$ in a countable language.  Let $S:= \{  M \in \Gamma \ | \ M \prec \mathfrak{A} \}$.  Then for every $M \in S$:\footnote{Note $S$ is stationary or contains a club, depending on whether we assume $\Gamma$ is $<ORD$ stationary or $<ORD$ club.}  
\begin{equation*}
\sigma_M^{-1}(b) \in H_M \subseteq A^{\Gamma}_{M \cap \kappa} = F^{\Gamma}(M \cap \kappa) 
\end{equation*}

Part \ref{item_SetCatchesAll} just follows from the class Pigeonhole Principle and the definition of $A^{\Gamma}_\alpha$ in (\ref{eq_A_alpha}).

Finally, suppose $\text{wCC}^*(\Gamma)$ fails.  Then $|F^{\Gamma}(\alpha)| < \alpha^+$ for almost every $\alpha < \kappa$.  Combined with item (\ref{item_SatisfiesGuessing}) of the current lemma, this implies that $F^{\Gamma}$ is a $\lavminus{\Gamma}$ function (in the case that $\Gamma$ was $<ORD$ stationary) or a $\lavplus{\Gamma}$ function (in the case that $\Gamma$ was $<ORD$-club).   
\end{proof}

\begin{corollary}\label{cor_SmallnessImpliesLaver}
Suppose $\kappa$ is regular uncountable and $\Gamma$ is a $< ORD$ stationary (resp. club) subclass of $\wp_\kappa(V)$ such that $\text{wCC}^*(\Gamma)$ fails.  Then $\lav{\Gamma}$ (resp. $\lavplus{\Gamma}$) holds.
\end{corollary}
\begin{proof}
This follows immediately from Lemma \ref{lem_ConstructLaver} and Theorem \ref{thm_GeneralizeKunenDiamondMinus}.
\end{proof}

The following remark shows that the converse of Corollary \ref{cor_SmallnessImpliesLaver} is false (however, Question \ref{q_Failure_wCC_responsible} asks whether a natural variation of the converse must hold).
\begin{remark}\label{rem_ConverseFails}
The converse of Corollary \ref{cor_SmallnessImpliesLaver} is false; i.e. it is possible for $\lav{\Gamma}$ to and $\text{wCC}(\Gamma)$ to simultaneously hold.  Suppose $\kappa$ is a supercompact cardinal, $\Gamma = \wp_\kappa(V)$, and $F:\kappa \to V_\kappa$ is a (classical) supercompact Laver function.  Then $\text{wCC}(\Gamma)$ holds.  

It is also possible to obtain a counterexample to the converse of Corollary \ref{cor_Failure_wCC_implies_Laver} where $\kappa$ is a successor cardinal.   Suppose $\kappa$ is supercompact and $F:\kappa \to V_\kappa$ is a (classic) Laver function; suppose also that $\kappa$ is almost huge.   Let $\mathbb{P}$ be any $\kappa$-cc forcing which is a subset of  $V_\kappa$ and turns $\kappa$ into a successor cardinal, and let $G$ be $(V,\mathbb{P})$-generic.  It is easy to see that the function $\tilde{F}:\kappa \to H_{\kappa}^{V[G]}$ defined by
\begin{equation*}
\alpha \mapsto \Big(F(\alpha)\Big)_{G \cap V_\alpha}
\end{equation*}
is a $\lav{\kappa}$ function in $V[G]$.\footnote{In fact $\tilde{F}$ is a $\lav{\Gamma}$ function for any class $\Gamma$ which has the property that for every $\theta$ there is a (possibly illfounded) generic embedding $i: V[G] \to W$ such that $i[H_\theta^{V[G]}] \in i(\Gamma)$.}  But the almost hugeness of $\kappa$ in $V$ implies that, in $V[G]$, the class $\wp_\kappa(V[G])$ has the $\text{wCC}$ property.
\end{remark}

Abstract Condensation principles have been extensively studied, for example by Law~\cite{MR2691107}, Woodin~\cite{MR1713438}, and Friedman-Holy~\cite{MR2860182}.  Friedman and Holy considered several versions of Condensation, and proved that Stationary Condensation (even stronger versions called Local Club Condensation) are consistent with $\kappa$ being a very large cardinal.  This is to be contrasted with the severe restraints that Club Condensation (as in \cite{MR1713438}) place on the large cardinal properties of $\kappa$.
\begin{definition}
Suppose $\kappa$ is a regular uncountable cardinal.  \emph{Stationary (resp. Club)  Condensation holds at $\kappa$} iff there exists an $\in$-increasing and $\subseteq$-continuous sequence $\langle M_\eta \ | \ \eta < \kappa \rangle$ of transitive sets such that,  letting
\begin{equation}\label{eq_CondenseTo}
\Gamma_{ \vec{M}}:= \{  X \in \wp_\kappa(V) \ | \ (\exists \eta < \kappa ) ( H_X = M_\eta )  \}
\end{equation}
then for every regular $\theta \ge \kappa$, the set $\Gamma_{ \vec{M}} \cap \wp_\kappa(H_\theta)$ is stationary (resp. Club).
\end{definition}
They also proved:
\begin{theorem}\label{thm_FriedmanHoly}[Friedman-Holy~\cite{MR2860182}]
Suppose $\kappa$ is regular.  Then there is a $\kappa$-directed closed class forcing extension which satisfies ZFC and Stationary Condensation at $\kappa$ (and all larger cardinals).
\end{theorem}
The forcing first forces GCH above $\kappa$ and then performs a reverse Easton iteration of adding Cohen subsets of cardinal successors above $\kappa$.

\begin{remark}
All of the results of this section actually hold using a weaker, non-linear form of condensation.  In particular, for the Stationary Condensation theorems about a successor cardinal $\kappa$, all we need is some stationary $\Gamma \subseteq \wp_\kappa(V)$ such that for most $\alpha < \kappa$, the map 
\begin{equation*}
s^\Gamma_\alpha:  \{ X \in \Gamma \ | \ X \cap \kappa = \alpha \} \to H_\kappa
\end{equation*}
defined by
\begin{equation*}
X \mapsto H_X
\end{equation*}
is at most $|\alpha|$-to-one.
\end{remark}

\begin{lemma}\label{lem_wCC_wCCstar_equiv}
Suppose $\vec{M} = \langle M_\eta \ | \ \eta < \kappa \rangle$ is an $\in$-increasing, $\subseteq$-continuous sequence of (transitive) sets.  Let $T \subseteq \Gamma_{ \vec{M}}$, where $\Gamma_{ \vec{M}}$ is defined as in (\ref{eq_CondenseTo}).  Then:
\begin{equation}\label{eq_Short_equiv_Small}
\text{wCC}(T) \iff  \text{wCC}^*(T)
\end{equation}
\end{lemma}
\begin{proof}
The $\Leftarrow$ direction of (\ref{eq_Short_equiv_Small}) is obvious, by Observation \ref{obs_SmallnessImpliesShortness}.  For the $\Rightarrow$ direction we show the contrapositive:  assume $\text{wCC}^*(T)$ holds, and consider a typical $\alpha < \kappa$ which witnesses this fact.   Since $\vec{M}$ is $\in$-increasing then the map $\eta \mapsto \text{height}(M_\eta)$ is one-to-one; it follows that for any pair $X,Y \in T$:
\begin{align*}
   \text{otp}(X \cap \theta) = \text{otp}(Y \cap \theta)   \implies H_X = H_Y
\end{align*}  
Thus since the set
\begin{equation*}
\{ H_X \ | \ X \in T \text{ and } X \cap \kappa = \alpha    \}
\end{equation*}
has cardinality $\ge \alpha^+$, then so does the set
\begin{equation*}
\{ \text{height}(H_X) \ | \ X \in T \text{ and } X \cap \kappa = \alpha    \}
\end{equation*}
which means that $\alpha$ witnesses that $\text{wCC}(T)$ holds.
\end{proof}

\begin{corollary}\label{cor_Failure_wCC_implies_Laver}
Assume Stationary Condensation (resp. Club Condensation) holds at $\kappa$, as witnessed by some $\vec{M} = \langle M_\eta \ | \ \eta < \kappa \rangle$.  Let $\Gamma \subseteq \wp_\kappa(V)$ be the class of structures that condense to $\vec{M}$.  Then:
\begin{equation*}
\neg \text{wCC}(\Gamma) \implies \lav{\Gamma} \text{   (resp. } \lavplus{\Gamma} \text{ )}
\end{equation*}
\end{corollary}
\begin{proof}
This follows directly from Corollary \ref{cor_SmallnessImpliesLaver} and Lemma \ref{lem_wCC_wCCstar_equiv}.
\end{proof}

\begin{lemma}\label{lem_StatCondenseSuccessorFail_CC}
If $\kappa$ is a \textbf{successor} cardinal and Stationary Condensation holds at $\kappa$ as witnessed by some $\vec{M} = \langle M_\eta \ | \ \eta < \kappa \rangle$, then $\text{wCC}^*(\Gamma_{\vec{M}})$ fails. 
\end{lemma}
\begin{proof}
Say $\kappa = \mu^+$.  Let $\alpha \in (\mu, \mu^+$); then there is some $\eta_\alpha < \kappa$ such that $M_{\eta_\alpha} \models \ \alpha \notin \text{CARD}$.\footnote{To see this:  let $f: \mu \to \alpha$ be surjective.  Let $X  \prec (H_\kappa, \in, \{ f \})$ with $\alpha \subset X$; then $X = H_X =  M_\eta$ for some $\eta$, and sees that $\alpha$ is not a cardinal.  }  If $X \prec_{\Sigma_1} (V,\in)$ and $\alpha = X \cap \kappa$ then $H_X \models \alpha \in \text{CARD}$.  It follows that
\begin{equation*}
R_\alpha:=\{  H_X \ | \  X \in \Gamma_{ \vec{M}} \text{ and } X \prec_{\Sigma_1} (V, \in) \text{ and }  X \cap \kappa = \alpha   \} \subset M_{\eta_\alpha}
\end{equation*}
and so
\begin{equation*}
|R_\alpha | \le | M_{\eta_\alpha}| < \kappa = \alpha^+
\end{equation*}
\end{proof}

\begin{corollary}
$L$ satisfies $\lavplus{\kappa}$ whenever $\kappa$ is a successor cardinal.
\end{corollary}
\begin{proof}
This follows from Lemma \ref{lem_StatCondenseSuccessorFail_CC} and Corollary  \ref{cor_Failure_wCC_implies_Laver}; note that $L$ satisfies Club Condensation.
\end{proof}

\section{Strengthening and simplification of main theorem from Viale~\cite{Viale_GuessingModel}}\label{sec_StrengthenViale}

The key ingredient of Viale's~\cite{Viale_GuessingModel} proof of that $PFA$ implies $\lav{\omega_2}$ is his ``Isomorphism Theorem" about a particular subclass of $\wp_{\omega_2}(V)$ which was shown by Viale-Wei\ss~\cite{VW_ISP} to be stationary in $\wp_{\omega_2}(H_\theta)$ for all $\theta$ (under the assumption of PFA).  This subclass of $\wp_{\omega_2}(V)$ is called the class of $\omega_1$-internally club, $\omega_1$-guessing models, and denoted $\text{GIC}_{\omega_1}$.  The stationarity of this class is responsible for much of the consistency strength and many of the consequences of PFA, and is widely conjectured to be equiconsistent with PFA.\footnote{Even without the ``internally club" part.}  We will not need to define $\text{GIC}_{\omega_1}$, but only use a few of its key properties.

For transitive $ZF^-$ models $H$ and $H'$, we say that \emph{$H$ is a hereditary initial segment of $H'$} iff either $H = H'$ or $H = (H_\lambda)^{H'}$ for some $\lambda \in \text{CARD}^{H'}$.  
\begin{definition}
A class \emph{$\Gamma$ has the $\kappa$-Isomorphism Property} iff whenever $X$, $X'$ are elements of $\Gamma$ and $X \cap \kappa = X' \cap \kappa$, then one of $H_X$, $H_{X'}$ is a hereditary initial segment of the other.
\end{definition}
The class $\text{GIC}_{\omega_1}$ easily projects downward (in the sense of Definition \ref{def_DownwardProjecting}),\footnote{See Lemma 10 of \cite{MR3096619}.} and Viale proved:\footnote{A simplified and elementary proof of the Isomorphism Theorem for $GIC_{\omega_1}$ can be found in Section 2.3 of Cox-Viale~\cite{MR3096619}.}
\begin{theorem}\label{thm_VialeIso}[Viale]
Assume $H$ and $H\rq{}$ are transitive $ZF^-$ models such that $H \cap H_{\omega_1} = H\rq{} \cap H_{\omega_1}$, and $H$, $H\rq{}$ are both in $GIC_{\omega_1}$.  Then one of $H$, $H\rq{}$ is a hereditary initial segment of the other.  
\end{theorem}
\begin{corollary}\label{cor_VialeIso}[Viale]
Suppose $\mathfrak{c} = \omega_2$ and $p: \omega_2 \leftrightarrow H_{\omega_1}$ is a bijection.  Then 
\begin{equation*}
\Gamma:= \bigcup_{\theta \ge \omega_2} \{   X \in \text{GIC}_{\omega_1} \ | \ X \prec (H_\theta, \in, p)           \}
\end{equation*}
satisfies the $\omega_2$-Isomorphism Property.
\end{corollary}
\begin{proof}
This follows immediately from Theorem \ref{thm_VialeIso}, since 
\begin{equation*}
H_X \cap H_{\omega_1} = p[X \cap \omega_2] = p[X' \cap \omega_2] = H_{X'} \cap H_{\omega_1}
\end{equation*}
\end{proof}
Note also that the class $\Gamma$ from Corollary \ref{cor_VialeIso} is $\Pi_1(V)$ definable from the parameter $p \in H_{\omega_3}$.  This parameter $p$ itself---i.e. this wellorder of $H_{\omega_2}$ in ordertype $\omega_2$---is not assumed to be definable in any way for our proof below of Theorem \ref{thm_StrengthenViale}.  Viale's original construction of a $\lav{\omega_2}$ function made use of a \emph{definable} wellorder of $H_{\omega_2}$ that exists under PFA, as proved by Caicedo-Velickovic~\cite{MR2231126}.  We show that this coding mechanism turns out to be unnecessary, and we also provide a simplified, direct construction of the Laver function.

\begin{remark}\label{rem_ISP_Weaker}
The conjunction of $\mathfrak{c}= \omega_2$ and stationarity of $\text{GIC}_{\omega_1}$ is strictly weaker than PFA.  It follows from PFA by Todorcevic~\cite{MR980949} and Viale-Wei\ss~\cite{VW_ISP}; but it does not imply PFA  as shown by two different constructions:\footnote{Each construction also shows that the conjunction of  $\mathfrak{c} = \omega_2$ with stationarity of $GIC_{\omega_1}$ does \textbf{not} imply the existence of the Caicedo-Velickovic coding.}
\begin{enumerate}
 \item  the author has shown that the model obtained by forcing with Neeman's~\cite{NeemanPFA} pure side condition poset using models of 2 types below a supercompact cardinal produces a model of $ \mathfrak{c} = \omega_2$ plus stationarity of $\text{GIC}_{\omega_1}$ which is not a model of PFA.
 \item  Menachem Magidor has shown that adding a Cohen real over an arbitrary model of PFA preserves the stationarity of $\text{GIC}_{\omega_1}$; and by Shelah~\cite{MR768264} this forcing extension does not even model Martin's Axiom.  
\end{enumerate}
\end{remark}

\subsection{Proof of Theorem \ref{thm_StrengthenViale}}

Suppose $\Gamma$ is a stationary subclass of $\wp_\kappa(V)$ satisfying the assumptions of Theorem \ref{thm_StrengthenViale}; i.e. it projects downward, it satisfies the $\kappa$-Isomorphism Property, and it is $\Pi_1(V)$ definable from some parameter $p$, where $p \in H_{\kappa^+}$.  Given a function $F:\kappa \to H_\kappa$, let us say that a set $b$ is a  \emph{witness to non-$\Gamma$-Laverness of $F$} iff there is some algebra $\mathcal{A}_b = (H_{|\text{trcl}(b)|^+}, \in, \{ p,b \}, \dots)$ such that $\sigma_M^{-1}(b) \ne F(M \cap \kappa)$ for every $M \in \Gamma \cap \wp_{\kappa}(H_\theta)$ such that $M \prec \mathcal{A}_b$.  We say that a regular cardinal $\theta$ is the \emph{least cardinal witnessing non-$\Gamma$-Laverness of $F$} iff $\theta$ is the least regular cardinal such that there is a $b \in H_\theta$ witnessing the non-$\Gamma$-Laverness of $F$.

For any $\alpha < \kappa$ and any partial $g: \alpha \to H_\kappa$, let $\mathcal{W}^{\Gamma}_{g}$ be the set of transitive $ZF^-$ models $W$ such that:
\begin{itemize}
 \item there is an elementary $\sigma:W \to_{\Sigma_1} V$ with $\alpha = \text{crit}(\sigma)$ and $\sigma(\alpha) = \kappa$;
 \item $\text{range}(\sigma) \in \Gamma$ and $p \in \text{range}(\sigma)$;
 \item $W$ has a largest cardinal $\theta_W$
 \item $g \in W$ and $W$ believes that $\theta_W$ is the least regular cardinal witnessing non-$\Gamma^W$-Laverness of $g$, where $\Gamma^W$ is the subset of $W$ defined over $W$ using the definition of $\Gamma$ and the parameter $\sigma^{-1}(p)$.  
\end{itemize}

\begin{nonGlobalClaim}\label{clm_AtMostOne}
For any $\alpha < \kappa$ and any $g: \alpha \to H_\kappa$, the set $\mathcal{W}^\Gamma_g$ has \textbf{at most one} element.
\end{nonGlobalClaim}
\begin{proof}[Proof of Claim \ref{clm_AtMostOne}]
Suppose $W$, $W'$ were two distinct elements of $\mathcal{W}^{\Gamma}_g$; let $\sigma:W \to_{\Sigma_1} V$ and $\sigma': W' \to_{\Sigma_1} V$ be the maps required by the definition of $\mathcal{W}^\Gamma_g$, and set $M:= \text{range}(\sigma)$ and $M':= \text{range}(\sigma')$.  Note that $p \in M \cap M' \cap H_{\kappa^+}$, and it follows easily (by coding $p$ as a subset of $\kappa$ in an absolute manner and using the assumption that $M \cap \kappa = M' \cap \kappa = \alpha$) that $p_W:= \sigma^{-1}(p)$ is equal to $p_{W'}:= \sigma^{' -1}(p)$.  By the assumption that $\Gamma$ has the $\kappa$-Isomorphism Property, one of $W$, $W'$ is a hereditary initial segment of the other; WLOG assume $W$ is a strict hereditary initial segment of $W'$.  Note then that $\theta_{W}$ (the largest cardinal of $W$) is strictly smaller than $\theta_{W'}$ (the largest cardinal of $W'$).  Note also, since $p_W = p_{W'}$, $\Gamma$ is $\Pi_1$ definable in $p$, and $W$ is a hereditary initial segment of (in particular a $\Sigma_1$ elementary substructure of) $W'$, then 
\begin{equation*}
\Gamma^W = \Gamma^{W'} \cap W
\end{equation*}
Let $\bar{H}:= (H_{\theta_W})^{W} = (H_{\theta_W})^{W'}$.  Then
\begin{equation*}
X:= (\wp_\alpha(\bar{H}))^W \cap \Gamma^W 
\end{equation*}  
Let $b \in \bar{H}$ witness (from the point of view of $W$) that $g$ is not a $\lav{\Gamma^W}$ function.  Then there is some algebra $\mathcal{A}_b \in W$ on $\bar{H}$ such that:
\begin{equation*}
W \models (\forall M \in X) \Big(  M \prec \mathcal{A}_b \implies \sigma_M^{-1}(b) \ne g(M \cap \alpha)  \Big)
\end{equation*}
This is a $\Sigma_1$ statement in the parameters $X$, $\mathcal{A}_b$, and $g$, and is thus upward absolute to $W'$; but since $X$ is also equal to $(\wp_\alpha(\bar{H}))^{W'} \cap \Gamma^{W'}$, this contradicts the minimality of $\theta_{W'}$.
\end{proof}

Now for any $\alpha < \kappa$ and any $g: \alpha \to H_\kappa$, let $W_g$ denote the unique element of $W^{\Gamma}_g$ given by Claim \ref{clm_AtMostOne} (if it exists) and let  $\theta_g$ the largest cardinal of $W_g$.  Let $b_g$ be \textbf{any} witness in  $(H_{\theta_g})^{W_g}$ to the non-Laverness of $g$ w.r.t. $\Gamma^{W_g}$.  

\begin{nonGlobalClaim}\label{clm_ExplicitConstructionLaver}
The function $F:\kappa \to H_\kappa$ defined recursively by $\alpha \mapsto b_{g \restriction \alpha}$ (if this exists; 0 otherwise) is a $\lav{\Gamma}$ function.  
\end{nonGlobalClaim}
\begin{proof}[Proof of Claim \ref{clm_ExplicitConstructionLaver}]
This is where we use assumption that $\Gamma$ projects downward, along with the assumption that $\Gamma$ is stationary at every $\wp_\kappa(H_\Omega)$.  Suppose for a contradiction that $F$ is not a $\lav{\Gamma}$ function; let $\theta$ be the least regular cardinal witnessing non-$\Gamma$-Laverness of $F$.  Let $\Omega:= \theta^+$ and, using stationarity of $\Gamma$, pick an $M' \prec (H_{\Omega}, \in, \{ F, p \})$ such that $M' \in \Gamma$.  Let $\alpha:= M' \cap \kappa$,  $\sigma_{M'}:H_{M'} \to M'$ be the inverse of the Mostowski collapse of $M'$, and $\bar{\Gamma} := \sigma_{M'}^{-1}(\Gamma \cap \wp_\kappa(H_\theta))$.  Then $\sigma_{M'}^{-1}(F) = F \restriction \alpha$ and $H_{M'}$ is the unique element of $\mathcal{W}^{\Gamma}_{F \restriction \alpha}$; so by the recursive definition of the function $F$, we know that $F(\alpha)$ is some element of $\bar{H}$ witnessing the non-$\bar{\Gamma}$-Laverness of $F \restriction \alpha$ from the point of view of $H_{M'}$.  Set $\bar{b}:= F(\alpha) \in H_{M'}$ and let $\bar{\mathcal{A}}=(\bar{H}, \in, \{ \bar{b}, \bar{p} \}, \dots  ) \in H_{M'}$ be an algebra corresponding to the witness $\bar{b}$.  Let $b:= \sigma_{M'}(\bar{b})$ and $\mathcal{A}:= \sigma_{M'}(\bar{\mathcal{A}})$; by elementarity of $\sigma_{M'}$:
\begin{equation}\label{eq_WhatH_OmegaBelieves}
H_{\Omega} \models (\forall M)\Big( M \prec \mathcal{A} \ \wedge \ M \in \Gamma \implies \sigma_M^{-1}(b) \ne F(M \cap \kappa) \Big)
\end{equation}    
Set $M:= M' \cap H_\theta$; by the downward projection assumption on the class $\Gamma$, we know that
\begin{equation*}
M \in \Gamma
\end{equation*}
Furthermore, since $M' \prec (H_\Omega, \in, \{ \mathcal{A} \})$ then
\begin{equation*}
M = M' \cap H_\theta \prec \mathcal{A}
\end{equation*}
Finally:
\begin{equation*}
\sigma_M^{-1}(b) = \sigma_{M'}^{-1}(b) = \bar{b} = F(\alpha) = F(M' \cap \kappa) = F(M \cap \kappa)
\end{equation*}
These properties of $M$ contradict (\ref{eq_WhatH_OmegaBelieves}), and complete the proof of the claim.
\end{proof}
This completes the proof of Theorem \ref{thm_StrengthenViale}.

\section{Concluding remarks}\label{sec_Conclusion}

Recall that Corollary \ref{cor_Failure_wCC_implies_Laver} said that if $\Gamma \subseteq \wp_\kappa(V)$ and the appropriate version of Chang's Conjecture fails for $\Gamma$, then $\lav{\Gamma}$ holds; and Remark \ref{rem_ConverseFails} demonstrated that the converse was not literally true.  However, it's still natural to wonder if $\lav{\Gamma}$ must always be \emph{essentially} due to some failure of Chang's Conjecture:
\begin{question}\label{q_Failure_wCC_responsible}
Suppose $\kappa$ is a successor cardinal, $\Gamma \subset \wp_\kappa(V)$, and $F:\kappa \to H_\kappa$ is a $\lav{\Gamma}$ function.  Must there be some (definable) $\Gamma' \subseteq \Gamma$ such that:
\begin{enumerate}
 \item $F$ is still a $\lav{\Gamma'}$ function; and
 \item $\text{wCC}(\Gamma')$ fails?
\end{enumerate}
\end{question}

\end{document}